\documentclass[11]{article}
\usepackage{amsmath, amssymb, graphicx, amsthm, mathrsfs}
\usepackage{mathptmx}
\usepackage[italic]{mathastext}
\usepackage{mathtools} 
\usepackage[utf8]{inputenc}
\usepackage{enumitem}  
\usepackage[colorlinks=true]{hyperref}
\usepackage{geometry}
\newtheorem{definition}{Definition}[section]
\newtheorem{theorem}[definition]{Theorem}
\newtheorem{lemma}[definition]{Lemma}
\newtheorem{remark}[definition]{Remark}
\newtheorem{example}[definition]{Example}

\newtheorem{corollary}[definition]{Corollary}

\numberwithin{equation}{definition}

\title{COMMON FIXED POINT THEOREMS OF WEAKLY COMPATIBLE MAPS SATISFYING  (f,g)-WEAKLY CONTRACTIVE CONDITION AND  INVARIANT APPROXIMATION RESULTS}

\author{Babu G.V.R\textsuperscript{1},  D. RATNA BATU \textsuperscript{2},  Alemayehu G. Negash\textsuperscript{3,*}\\
\small \textsuperscript{1}Department of Mathematics, Andhra University, India\\
\small \textsuperscript{1}Email: \texttt{gvr\_babu@hotmail.com} 
\\
\small \textsuperscript{2}Department of Mathematics, PSCMR College of Engineering and Technology, Vijayawada,
Andhra Pradesh, India\\
\small \textsuperscript{2}Email: ratnababud@gmail.com\\
\small \textsuperscript{3}Department of Mathematics, Hampton University, USA\\
$^*$ \textit{Corresponding author.} \\
\small \textsuperscript{3}Email: \texttt{alemayehu.negash@hamptonu.edu}\
}

\date{}

\begin{document}

\maketitle

\begin{abstract}
We prove the existence of common fixed points for three selfmaps $T,f$ and $g$ defined on a metric space $(X,d)$ satisfying, $T$ is $(f,g)$-weakly contractive; and the pairs $(T,f)$ and $(T,g)$ are weakly compatible. Also, for such $T,f$ and $g$, we prove the convergence of modified Mann iteration and modified Ishikawa iteration with respect to $T,f$ and $g$ to their common fixed point, in a normed space. Further, we obtain invariant approximation results from the set of best approximations to common fixed points of $T,f$ and $g$.
\end{abstract}


\textbf{AMS(2000) Mathematics Subject Classification:} 47H10, 54H25.

\emph{SHORT RUNNING TITLE:} COMMON FIXED POINT THEOREMS OF WEAKLY COMPATIBLE \ldots

\textsuperscript{*}Corresponding author

\section{Introduction}

Let $(X,d)$ be a metric space. We write it simply as $X$. A map $T:X\to X$ is said to be a selfmap of $X$.

\begin{definition}
A selfmap $T$ of $X$ is said to be a \emph{contraction} if there exists a real number $k\in[0,1)$ such that
\end{definition}

\begin{equation}
d(Tx,Ty)\leq k\ d(x,y),\ \text{for all}\ x,y\in X.
\end{equation}

If $k=1$ then we call $T$ a \emph{nonexpansive map}.

In 1997, Alber and Guerre-Delabriere \cite{Alber1997} introduced weakly contractive maps which are extensions of contraction maps and obtained fixed point results in the setting of Hilbert spaces.

Throughout this paper, we denote $R_{+}=[0,\infty)$,

$\Psi\ =\ \{\psi/\psi:R_{+}\to R_{+}$ is continuous on $R_{+}$, $\psi$ is nondecreasing,

\[\psi(t)>0\ \ \text{for}\ \ t>0,\ \psi(0)=0\ \text{and}\ \lim_{t\to\infty}\psi(t )=\infty\}\]

and $Z_{+}\ =\ \{0,1,2,\ldots\}$.

\begin{definition}[\cite{Rhoades2001}]
Let $(X,d)$ be a metric space. Let $K$ be a nonempty subset of $X$. A selfmap $T$ of $K$ is called \emph{weakly contractive} if, there exists $\psi\in\Psi$ such that
\[d(Tx,Ty)\leq d(x,y)-\psi(d(x,y))\ \ \text{for each}\ \ x,y\in K.\]
\end{definition}

We observe that every contraction map is weakly contractive map. But its converse need not be true; i.e., every weakly contractive map need not be a contraction.

\begin{example}
Let $X=R_{+}$, with the usual metric. Define $T:R_{+}\to R_{+}$ by $T(x)=\frac{x}{1+x},x\in X$.

Then $T$ is a weakly contractive map with $\psi(t)=\frac{t^{2}}{1+t},\ t\geq 0$.

But, for any $k\in[0,1)$, $T$ does not satisfy (1.1.1)

when $x=0$ and $0<y<\frac{1-k}{k}$ so that $T$ is not a contraction map.
\end{example}

An interesting result for a single weakly contractive selfmap in complete metric spaces is established by Rhoades (Theorem 1, page no. 2684, \cite{Rhoades2001}), which we call `Rhoades weakly contractive principle'.

\begin{theorem}[\textbf{Rhoades weakly contractive principle}]
Let $(X,d)$ be a complete metric space, $T$ a weakly contractive map. Then $T$ has a unique fixed point in $X$.
\end{theorem}

In 2006, Beg and Abbas \cite{Beg2006} extended the concept of weakly contractive property of a single selfmap to a pair of selfmaps and established the existence of common fixed points.

\begin{definition}
Let $(X,d)$ be a metric space. Let $T:X\to X$ and $f:X\to X$ be two maps. We say that $T$ \emph{is weakly contractive with respect to $f$} if there exists a $\psi\in\Psi$ such that
\begin{equation}
d(Tx,Ty)\leq d(fx,fy)-\psi(d(fx,fy)),\text{ for each }x,y\in X.
\end{equation}
\end{definition}

\begin{theorem}[\textbf{Beg and Abbas \cite{Beg2006}, Theorem 2.5, page 4}]
Let $(X,d)$ be a metric space and let $T$ be a weakly contractive mapping with respect to $f$. If $T$ and $f$ are weakly compatible and $T(X)\subseteq f(X)$ and $f(X)$ is a complete subspace of $X$, then $f$ and $T$ have a unique common fixed point in $X$.
\end{theorem}

In 2008, Akbar Azam and Muhammad Shakeel \cite{Azam2008} proved the following theorem, which is an extension of Theorem 1.6 to three selfmaps.

\begin{theorem}[\textbf{A. Azam and M. Shakeel \cite{Azam2008}, Theorem 2.4, page 104}]
Let $X$ be a metric space and $T,f,g$ be selfmaps of $X$ and for each $x,y\in X$, satisfying
\begin{equation}
d(gx,Ty)\leq d(fx,fy)-\varphi(d(fx,fy)),
\end{equation}
where $\psi$ is continuous, nondecreasing and positive on $(0,\infty)$, $\psi(0)=0$ and $\lim\limits_{t\rightarrow\infty}\psi(t)=\infty$. If the pairs $(f,g)$ and $(f,T)$ are weakly compatible and $T(X)\cup g(X)\subset f(X)$ and $f(X)$ is a complete subspace of $X$, then $T,f$ and $g$ have a common fixed point.
\end{theorem}

In this paper, we extend the concept of weakly contractive property introduced in Definition 1.5 to three selfmaps in the following way.

\begin{definition}
Let $(X,d)$ be a metric space. Let $T$, $f$ and $g$ be three selfmaps of $X.$ A map $T$ is said to be $(f,g)-$\emph{weakly contractive}, if there exists a $\psi\in\Psi$ such that
\begin{equation}
d(Tx,Ty)\leq\min\{d(fx,gy)-\psi(d(fx,gy)),d(gx,fy)-\psi(d(gx,fy))\}
\end{equation}
for each $x,y\in X.$
\end{definition}

The following two examples suggest that the inequalities (1.7.1) and (1.8.1) are independent.

\begin{example}
Let $X=(\frac{1}{4},1]$ with the usual metric.

Define $T,f$ and $g$ on $X$ by

\[
T(x)=
\begin{cases}
\frac{1}{2}, & \text{if } \frac{1}{4}<x<\frac{2}{3},\\
1-\frac{1}{2}x, & \text{if } \frac{2}{3}\le x\le 1,
\end{cases}
\qquad
f(x)=
\begin{cases}
1, & \text{if } \frac{1}{4}<x<\frac{2}{3},\\
\frac{4}{3}-x, & \text{if } \frac{2}{3}\le x\le 1,
\end{cases}
\quad \text{and} \quad
g(x)=
\begin{cases}
\frac{1}{3}, & \text{if } \frac{1}{4}<x<\frac{2}{3},\\
\frac{4}{3}-x, & \text{if } \frac{2}{3}\le x\le 1.
\end{cases}
\]
Define \(\psi:\mathbb{R}_{+}\to \mathbb{R}_{+}\) by \(\psi(t)=\frac{t^{2}}{1+t}\), \(t\in \mathbb{R}_{+}.\) Then these \(T,f\) and \(g\) satisfy (1.8.1). Now, \(x,y\in(\frac{1}{4},\frac{2}{3})\).\\
Consider \(d(gx,Ty)=\frac{1}{6}\).
But \(d(fx,fy)-\psi(d(fx,fy))=0\), so that \(T,\ f\) and \(g\) do not satisfy the inequality (1.7.1), for any \(\psi\in\Psi\).
\end{example}

\textbf{Example 1.10.} Let \(X=(0,1]\) with the usual metric.
Define \(T,\ f\) and \(g\) on \(X\) by

\[
T(x)=
\begin{cases}
\frac{1}{2}, & \text{if } 0<x<\frac{2}{3},\\
1-\frac{1}{2}x, & \text{if } \frac{2}{3}\le x<1,\\
\frac{2}{3}, & \text{if } x=1,
\end{cases}
\qquad
f(x)=
\begin{cases}
\frac{1}{3}, & \text{if } 0<x<\frac{2}{3},\\
\frac{4}{3}-x, & \text{if } \frac{2}{3}\le x<1,\\
1, & \text{if } x=1,
\end{cases}
\quad \text{and} \quad
g(x)=
\begin{cases}
\frac{1}{2}, & \text{if } 0<x<\frac{2}{3},\\
1-\frac{1}{2}x, & \text{if } \frac{2}{3}\le x\le 1.
\end{cases}
\]
Then clearly \(g(X)\cup T(X)\subseteq f(X)\).\\
Define \(\psi:\mathbb{R}_{+}\to \mathbb{R}_{+}\) by \(\psi(t)=\frac{t}{2},\ t\in \mathbb{R}_{+}.\) Then it is easy to see that \(T,f\) and \(g\) satisfy the inequality (1.7.1).\\
But for \(x=\frac{2}{3}\) and \(y=\frac{5}{6}\),

\(d(Tx,Ty)=\frac{1}{6};\) and \(d(fx,gy)-\psi(d(fx,gy))=\frac{1}{12}-\psi(\frac{1}{12}),\) so that the inequality (1.8.1) does not hold for any \(\psi\).\\
In 2006, Beg and Abbas [3] introduced modified Mann and modified Ishikawa iterations associated with two selfmaps \(T\) and \(f\), and established the convergence results for weakly contractive and weakly compatible maps.

\textbf{Definition 1.11.} Let \(X\) be a Banach space. Assume that \(T(X)\subseteq f(X)\) and \(f(X)\) is a convex subset of \(X\). Define a sequence \(\{y_{n}\}\) in \(f(X)\) as

\[
y_{n}=f(x_{n+1})=(1-\alpha_{n})f(x_{n})+\alpha_{n}Tx_{n},\ x_{0}\in X,\ n\geq 0, \tag{1.11.1}
\]
where \(0\leq\alpha_{n}\leq 1\) for each \(n\). The sequence \(\{y_{n}\}\) thus obtained is called modified Mann iterative scheme.

\textbf{Definition 1.12.} 
Let \(X\) be a Banach space. Assume that \(T(X) \subseteq f(X)\) and that \(f(X)\) is a convex subset of \(X\). 
For \(x_{0} \in X\), define two sequences \(\{y_{n}\}\) and \(\{z_{n}\}\) in \(f(X)\) by
\begin{align}
z_{n} &= f(x_{n+1}) = (1 - \alpha_{n}) f(x_{n}) + \alpha_{n} T(v_{n}), \nonumber\\
y_{n} &= f(v_{n}) = (1 - \beta_{n}) f(x_{n}) + \beta_{n} T(x_{n}), \quad n \ge 0, \tag{1.12.1}
\end{align}
where \(0 \le \alpha_{n}, \beta_{n} \le 1\) for each \(n\). 
The sequence \(\{z_{n}\}\) thus obtained is called a \emph{modified Ishikawa iterative scheme}.

We define modified Mann and modified Ishikawa iteration schemes with respect to \(T,f\) and \(g\) as follows.

\textbf{Definition 1.13.} 
Let \(X\) be a Banach space. Suppose that 
\(T(X) \subseteq f(X)\) and \(T(X) \subseteq g(X)\). 
Assume that \(f(X)\) and \(g(X)\) are convex subsets of \(X\). 
For \(x_{0} \in X\), define a sequence \(\{y_{n}\}\) by
\begin{align}
y_{2n}   &= f(x_{2n+1})   = (1 - \alpha_{2n})f(x_{2n})   + \alpha_{2n}T(x_{2n}), \nonumber\\
y_{2n+1} &= g(x_{2n+2})   = (1 - \alpha_{2n+1})g(x_{2n+1}) + \alpha_{2n+1}T(x_{2n+1}), \tag{1.13.1}
\end{align}
for \(n \ge 0\), where \(0 \le \alpha_{n} \le 1\) for each \(n\).  
The sequence \(\{y_{n}\}\) thus obtained is called a \emph{modified Mann iterative scheme} with respect to \(T\), \(f\), and \(g\).

Here we observe that if \(g=f\) in (1.13.1) then we obtain (1.11.1).

\textbf{Definition 1.14.} 
Let \(X\) be a Banach space. Suppose that 
\(T(X) \subseteq f(X)\) and \(T(X) \subseteq g(X)\). 
Assume that \(f(X)\) and \(g(X)\) are convex subsets of \(X\). 
For \(x_{0} \in X\), define two sequences \(\{z_{n}\}\) and \(\{y_{n}\}\) by
\begin{align}
z_{2n}   &= f(x_{2n+1})   = (1 - \alpha_{2n})f(x_{2n})   + \alpha_{2n}T(v_{2n}), \nonumber\\
z_{2n+1} &= g(x_{2n+2})   = (1 - \alpha_{2n+1})g(x_{2n+1}) + \alpha_{2n+1}T(v_{2n+1}), \nonumber\\
y_{2n}   &= f(v_{2n})     = (1 - \beta_{2n})f(x_{2n})     + \beta_{2n}T(x_{2n}), \nonumber\\
y_{2n+1} &= g(v_{2n+1})   = (1 - \beta_{2n+1})g(x_{2n+1}) + \beta_{2n+1}T(x_{2n+1}), \tag{1.14.1}
\end{align}
for \(n \ge 0\), where \(0 \le \alpha_{n}, \beta_{n} \le 1\) for each \(n\).  
The sequence \(\{z_{n}\}\) thus defined is called a \emph{modified Ishikawa iterative scheme} with respect to \(T\), \(f\), and \(g\).\\
Note that the iterations (1.14.1) implies (1.12.1) if \(g=f\).

In section 2, we prove the existence of common fixed points for three selfmaps $T,f$ and $g$ defined on a metric space $(X,d)$ satisfying, $T$ is $(f,g)$-weakly contractive; and the pairs $(T,f)$ and $(T,g)$ are weakly compatible. Also, for such $T,f$ and $g$, we prove the convergence of modified Mann iteration and modified Ishikawa iteration with respect to $T,f$ and $g$ to their common fixed point, in a normed space. In section 3, we obtain invariant approximation results from the set of best approximations to common fixed points of $T,f$ and $g$.

\section{Main Results}

The following theorem is our main result.

\begin{theorem}
Let $(X,d)$ be a metric space and let $T$ be a $(f,g)-$ weakly contractive mapping. If the pairs $(T,f)$ and $(T,g)$ are weakly compatible and $\overline{T(X)}\subseteq f(X)$, $\overline{T(X)}\subseteq g(X)$ and $\overline{T(X)}$ is a complete subspace of $X$, then $T$, $f$ and $g$ have a unique common fixed point in $X$.
\end{theorem}

\begin{proof}
Let $x_{0}\in X.$ Since $\overline{T(X)}\subseteq f(X)$, $T(X)\subseteq f(X).$ So there exists $x_{1}\in X$ such that $Tx_{0}=fx_{1}.$
Since $\overline{T(X)}\subseteq g(X)$, $T(X)\subseteq g(X).$ So there exists $x_{2}\in X$ such that $Tx_{1}=gx_{2}.$
On continuing this process, inductively we get a sequence $\{x_{n}\}$ in $X$ such that $y_{2n}=T(x_{2n})=f(x_{2n+1})$ and $y_{2n+1}=T(x_{2n+1})=g(x_{2n+2})$, $n=0,1,2,\ldots.$
We now show that the sequence $\{d(y_{n},y_{n+1})\}$ is a decreasing sequence. Now,
\begin{align*}
d(y_{2n},y_{2n+1}) &= d(T(x_{2n}),T(x_{2n+1})) \\
&\leq \min\{d(f(x_{2n}),g(x_{2n+1}))-\psi(d(f(x_{2n}),g(x_{2n+1}))), \\
&\quad d(g(x_{2n}),f(x_{2n+1}))-\psi(d(g(x_{2n}),f(x_{2n+1})))) \\
&\leq d(g(x_{2n}),f(x_{2n+1}))-\psi(d(g(x_{2n}),f(x_{2n+1}))) \\
&= d(y_{2n-1},y_{2n})-\psi(d(y_{2n-1},y_{2n})) \\
&< d(y_{2n-1},y_{2n}).
\end{align*}
Therefore
\begin{equation}
d(y_{2n},y_{2n+1})<d(y_{2n-1},y_{2n}).
\end{equation}
Similarly, it is easy to see that
\begin{equation}
d(y_{2n+1},y_{2n+2})<d(y_{2n},y_{2n+1}).
\end{equation}
Thus from (2.1.1) and (2.1.2) it follows that

\[d(y_{n},y_{n+1})<d(y_{n-1},y_{n}),\ n=0,1,2,\ldots.\]
Therefore $\{d(y_{n},y_{n+1})\}$ is a strictly decreasing sequence of nonnegative reals and hence convergent and let its limit be $l\geq 0$.
If $l>0$ then $l=\inf\limits_{n}d(y_{n},y_{n+1})$.
Thus we have $l\leq d(y_{n},y_{n+1})$ for all $n$.\\
Now,
\begin{align}
d(y_{n},y_{n+1}) &= d(T(x_{n}),T(x_{n+1})) \nonumber \\
&\leq \min\{d(f(x_{n}),g(x_{n+1}))-\psi(d(f(x_{n}),g(x_{n+1}))), \nonumber \\
&\quad d(g(x_{n}),f(x_{n+1}))-\psi(d(g(x_{n}),f(x_{n+1})))\}.
\end{align}
If $n$ is odd then from (2.1.3), we get
\begin{align}
d(y_{n},y_{n+1}) &= d(T(x_{n}),T(x_{n+1})) \nonumber \\
&\leq d(f(x_{n}),g(x_{n+1}))-\psi(d(f(x_{n}),g(x_{n+1}))) \nonumber \\
&= d(y_{n-1},y_{n})-\psi(d(y_{n-1},y_{n})).\nonumber
\end{align}
Therefore
\begin{equation}
    d(y_{n},y_{n+1})\leq d(y_{n-1},y_{n})-\psi(d(y_{n-1},y_{n})).
\end{equation}
If $n$ is even then from (2.1.3), we get
\begin{align}
d(y_{n},y_{n+1}) &= d(T(x_{n}),T(x_{n+1})) \nonumber \\
&\leq d(g(x_{n}),f(x_{n+1}))-\psi(d(g(x_{n}),f(x_{n+1}))) \nonumber \\
&= d(y_{n-1},y_{n})-\psi(d(y_{n-1},y_{n})).\nonumber
\end{align}
Therefore
\begin{equation}
    d(y_{n},y_{n+1})\leq d(y_{n-1},y_{n})-\psi(d(y_{n-1},y_{n})).
\end{equation}
We have, from (2.1.4) and (2.1.5)
\begin{equation}
d(y_{n},y_{n+1})\leq d(y_{n-1},y_{n})-\psi(d(y_{n-1},y_{n}))\ \ \text{for all}\ \ n.
\end{equation}
On taking limit as $n\to\infty$ in (2.1.6), we get
\[
\lim\limits_{n\to\infty}d(y_{n},y_{n+1})\leq\lim\limits_{n\to\infty}d(y_{n-1},y_{n})-\psi(\lim\limits_{n\to\infty}d(y_{n-1},y_{n})).
\]
Therefore $l\leq l-\psi(l)<l$, a contradiction.
Hence, $l=0$.\\
Thus it follows that
\[
\lim_{n\to\infty}d(y_n, y_{n+1}) = 0.
\]
\textbf{Case(i):}  $y_n = y_{n+1}$ for some $n$.\\
If $n$ is even, we write $n = 2m$ (say), $m \in Z_+$.
Thus, from (1.8.1) it follows that
\[
y_{2m} = y_{2m+k} \quad \text{for all } k = 0, 1, 2, \ldots.
\]
If $n$ is odd then $n = 2m + 1$ (say), $m \in Z_+$.
From (1.8.1), we get
\[
y_{2m+1} = y_{2m+1+k} \quad \text{for all } k = 0, 1, 2, \ldots.
\]
Therefore
\[
y_n = y_{n+k} \quad \text{for all } k = 0, 1, 2, \ldots.
\]
Thus it follows that $\{y_m\}_{m\geq n}$ is a constant sequence and hence it is a Cauchy sequence.\\
\textbf{Case(ii):}  $y_n \neq y_{n+1}$ for all $n$.\\
We now show that $\{y_n\}$ is a Cauchy sequence in $T(X)$.
It is sufficient to show that the sequence $\{y_{2n}\}$ is a Cauchy sequence.
Suppose $\{y_n\}$ is not Cauchy. Then there exists $\epsilon > 0$ for all $k \in \mathbb{N}$ there exist $\{m_k\}$, $\{n_k\}$ with $n_k > m_k > k$ such that
\begin{equation}
d(y_{2m_k}, y_{2n_k}) \geq \epsilon \quad \text{and} \quad d(y_{2m_k}, y_{2n_k-2}) < \epsilon.
\end{equation}
Now,
\begin{equation}
d(y_{2n_k}, y_{2m_k}) \leq d(y_{2n_k}, y_{2n_k-1}) + d(y_{2n_k-1}, y_{2n_k-2}) + d(y_{2n_k-2}, y_{2m_k}).
\end{equation}
On taking limits as $k \to \infty$ in (2.1.8) and using (2.1.7), we get
\[
\epsilon \leq \lim_{k\to\infty}d(y_{2n_k}, y_{2m_k}) + \lim_{k\to\infty}d(y_{2n_k-1}, y_{2n_k-2}) + \lim_{k\to\infty}d(y_{2n_k-2}, y_{2m_k}) \leq \epsilon.
\]
Therefore
\begin{equation}
\lim_{k\to\infty}d(y_{2n_k}, y_{2m_k}) = \epsilon.
\end{equation}
We now show that $\lim\limits_{k\to\infty}d(y_{2m_k-1}, y_{2n_k}) = \epsilon$.
Now,
\begin{equation}
d(y_{2m_k-1}, y_{2n_k}) \leq d(y_{2m_k-1}, y_{2n_k-1}) + d(y_{2n_k-1}, y_{2n_k}).
\end{equation}
On taking limits as $k\to\infty$ in (2.1.10), we get
\begin{align*}
\lim_{k\to\infty}d(y_{2m_{k}-1},y_{2n_{k}}) &\leq \lim_{k\to\infty}d(y_{2m_{k}-1},y_{2n_{k}-1})+\lim_{k\to\infty}d(y_{2n_{k}-1},y_{2n_{k}}) \\
&\leq \epsilon.
\end{align*}
Therefore
\begin{equation}
\lim_{k\to\infty}d(y_{2m_{k}-1},y_{2n_{k}}) \leq \epsilon.
\end{equation}
Now,
\begin{align}
\epsilon &\leq d(y_{2m_{k}-1},y_{2n_{k}-1}) \nonumber\\
&\leq d(y_{2m_{k}-1},y_{2n_{k}})+d(y_{2n_{k}},y_{2n_{k}-1})
\end{align}
On taking limits as $k\to\infty$ in (2.1.12), we get
\begin{equation}
\epsilon\leq\lim_{k\to\infty}d(y_{2m_{k}-1},y_{2n_{k}}).
\end{equation}
We have, from (2.1.11) and (2.1.13)
\begin{equation}
\lim_{k\to\infty}d(y_{2m_{k}-1},y_{2n_{k}})=\epsilon.
\end{equation}
Thus by using (2.1.9) and (2.1.14), we get
\begin{align*}
\epsilon &= d(y_{2m_{k}},y_{2n_{k}}) \\
&\leq d(y_{2m_{k}},y_{2n_{k}+1})+d(y_{2n_{k}+1},y_{2n_{k}}) \\
&= d(T(x_{2m_{k}}),T(x_{2n_{k}+1}))+d(y_{2n_{k}+1},y_{2n_{k}}) \\
&\leq \min\{d(f(x_{2m_{k}}),g(x_{2n_{k}+1}))-\psi(d(f(x_{2m_{k}}),g(x_{2n_{k}+1}))), \\
&\quad d(g(x_{2m_{k}}),f(x_{2n_{k}+1}))-\psi(d(g(x_{2m_{k}}),f(x_{2n_{k}+1})))\} \\
&\quad +d(y_{2n_{k}+1},y_{2n_{k}}) \\
&\leq d(g(x_{2m_{k}}),f(x_{2n_{k}+1}))-\psi(d(g(x_{2m_{k}}),f(x_{2n_{k}+1}))) \\
&\quad +d(y_{2n_{k}+1},y_{2n_{k}}) \\
&= d(y_{2m_{k}-1},y_{2n_{k}})-\psi(d(y_{2m_{k}-1},y_{2n_{k}})))+d(y_{2n_{k}+1},y_{2n_{k}}).
\end{align*}
Therefore
\begin{equation}
\epsilon \leq d(y_{2m_{k}-1},y_{2n_{k}})-\psi(d(y_{2m_{k}-1},y_{2n_{k}}))+d(y_{2n_{k}+1},y_{2n_{k}}).
\end{equation}
On taking limits as $k\to\infty$ in (2.1.15), we get
\begin{align*}
\epsilon &\leq \lim_{k\to\infty}d(y_{2m_{k}-1},y_{2n_{k}})-\lim_{k\to\infty}\psi(d(y_{2m_{k}-1},y_{2n_{k}}))+\lim_{k\to\infty}d(y_{2n_{k}+1},y_{2n_{k}}) \\
&= \epsilon-\psi(\epsilon)<\epsilon,
\end{align*}
a contradiction.
Therefore $\{y_{2n}\}$ is a Cauchy sequence in $T(X)$ and hence $\{y_{n}\}$ is Cauchy in $T(X)$.\\
Since $T(X)\subseteq\overline{T(X)}$, we have $\{y_{n}\}$ is Cauchy in $\overline{T(X)}$ and $\overline{T(X)}$ is complete there exists $z\in T(X)$ such that $y_{n}\to z$ as $n\to\infty$.\\
Now,
\[
\lim_{n\to\infty}y_{2n}=\lim_{n\to\infty}T(x_{2n})=\lim_{n\to\infty}f(x_{2n+1})=z
\]
and
\[
\lim_{n\to\infty}y_{2n+1}=\lim_{n\to\infty}T(x_{2n+1})=\lim_{n\to\infty}g(x_{2n+2})=z.
\]
Since $\overline{T(X)}\subseteq f(X)$ and $z\in\overline{T(X)}$, there exists $u\in X$ such that $z=f(u)$ and since $T(X)\subseteq g(X)$ and $z\in T(X)$, there exists $v\in X$ such that $z=g(v)$.\\
We show now that $z=T(u)$ and $z=T(v)$.\\
Now,
\begin{align}
d(Tu,T(x_{2n+2})) &\leq \min\{d(fu,g(x_{2n+2}))-\psi(d(fu,g(x_{2n+2}))), \nonumber \\
&\quad d(gu,f(x_{2n+2}))-\psi(d(gu,f(x_{2n+2})))\} \nonumber \\
&\leq d(fu,g(x_{2n+2}))-\psi(d(fu,g(x_{2n+2}))).
\end{align}
On taking limits as $n\to\infty$ in (2.1.16), we get
\[
d(Tu,z) \leq d(z,z)-\psi(d(z,z)) = 0.
\]
Thus it follows that $Tu=z$.\\
Now,
\begin{align}
d(Tv,T(x_{2n+1})) &\leq \min\{d(fv,g(x_{2n+1}))-\psi(d(fv,g(x_{2n+1}))), \nonumber \\
&\quad d(gv,f(x_{2n+1}))-\psi(d(gv,f(x_{2n+1})))\} \nonumber \\
&\leq d(gv,f(x_{2n+1}))-\psi(d(gv,f(x_{2n+1}))).
\end{align}
On taking limits as $n\to\infty$ in (2.1.17), we get
\[
d(Tv,z) \leq d(z,z)-\psi(d(z,z)) = 0.
\]
Thus it follows that $Tv=z$.\\
Since $(T,f)$ and $(T,g)$ are weakly compatible and since $Tu=fu$ and $Tv=gv$ we have
\[
fz=f(Tu)=T(fu)=Tz=T(gv)=g(Tv)=gz.
\]
This shows that $z$ is a coincidence point of $T,f$ and $g$.\\
We now show that $Tz=z$.\\
Consider
\begin{align*}
d(Tz,z) &= d(Tz,Tu) \\
&\leq \min\{d(fz,gu)-\psi(d(fz,gu)),d(gz,fu)-\psi(d(gz,fu))\} \\
&\leq d(gz,fu)-\psi(d(gz,fu)) \\
&= d(Tz,z)-\psi(d(Tz,z))
\end{align*}
Therefore $d(Tz,z)\leq d(Tz,z)-\psi(d(Tz,z))$, which implies that $\psi(d(Tz,z))\leq 0$. This shows that $d(Tz,z)=0$ and hence $Tz=z$.\\
Therefore $z$ is a common fixed point of $T,f$ and $g$.

\subsection*{Uniqueness of $z$:}

Let $t$ be another common fixed point of $T,f$ and $g$ with $t\neq z$.
Then
\begin{align*}
d(t,z) &= d(Tt,Tz) \\
&\leq \min\{d(ft,gz)-\psi(d(ft,gz)),d(gt,fz)-\psi(d(gt,fz))\} \\
&\leq d(ft,gz)-\psi(d(ft,gz)) \\
&= d(t,z)-\psi(d(t,z)).
\end{align*}
Thus it follows that $\psi(d(t,z))\leq 0$ so that $d(t,z)=0$ and hence $t=z$.\\
Therefore $z$ is the unique common fixed point of $T,f$ and $g$.\\
Hence the theorem follows.
\end{proof}

\begin{corollary}
Let $\{T_{n}\}_{n=1}^{\infty},f$ and $g$ be selfmaps on a metric space $(X,d)$.
Assume that $\overline{T_{1}(X)}\subseteq f(X),\overline{T_{1}(X)}\subseteq g(X)$ and $\overline{T_{1}(X)}$ is complete and
\begin{equation}
d(T_{1}x,T_{j}y)\leq\min\{d(fx,gy)-\psi(d(fx,gy)),d(gx,fy)-\psi(d(gx,fy))\}
\end{equation}
for all $x,y\in X$ and $j=1,2,3,\ldots$. If the pairs $(T_{1},f)$ and $(T_{1},g)$ are weakly compatible, then $\{T_{n}\}_{n=1}^{\infty}$, $f$ and $g$ have a unique common fixed point in $X$.
\end{corollary}

\begin{proof}
By taking $j=1$ in (2.2.1) then, we have
\[
d(T_{1}x,T_{1}y)\leq\min\{d(fx,gy)-\psi(d(fx,gy)),d(gx,fy)-\psi(d(gx,fy))\}
\]
so that $T_{1}$ is $(f,g)-$weakly contractive map.
Thus by Theorem 2.1, $T_{1}$, $f$ and $g$ have a unique common fixed point in $X$
say $z$.\\
Now, let $j\in N$ with $j\neq 1$ then, from (4.1.1) we have
\begin{align*}
d(T_{1}z,T_{j}z) &\leq \min\{d(fz,gz)-\psi(d(fz,gz)),d(gz,fz)-\psi(d(gz,fz))\} \\
&\leq d(z,z)-\psi(d(z,z)) = 0.
\end{align*}
Therefore
\[
d(z,T_{j}z)\leq 0\;\text{ and hence }T_{j}z=z.
\]
Thus $z$ is a unique common fixed point of the sequence $\{T_{n}\}_{n=1}^{\infty},\;f\text{ and }g.$\\
Hence the corollary follows.
\end{proof}
The following is an example in support of Theorem 2.1.

\begin{example}
Let $X=(\frac{1}{4},1]$ with the usual metric. Define $T$, $f$ and $g$ on $X$ by

\[
T(x)=
\begin{cases}
\frac{1}{2}, & \text{if } \frac{1}{4}<x<\frac{2}{3},\\
1-\frac{1}{2}x, & \text{if } \frac{2}{3}\le x\le 1,
\end{cases}
\qquad
f(x)=
\begin{cases}
1, & \text{if } \frac{1}{4}<x<\frac{2}{3},\\
\frac{4}{3}-x, & \text{if } \frac{2}{3}\le x\le 1,
\end{cases}
\quad \text{and} \quad
g(x)=
\begin{cases}
\frac{1}{3}, & \text{if } \frac{1}{4}<x<\frac{2}{3},\\
\frac{4}{3}-x, & \text{if } \frac{2}{3}\le x\le 1.
\end{cases}
\]
Then $\overline{T(X)}\subseteq f(X),\;T(X)\subseteq g(X),\;(T,f)$ and $(T,g)$ are weakly compatible and $\frac{2}{3}$ is the unique common fixed point of $T,f$ and $g$.\\
Define $\psi:R_{+}\to R_{+}$ by $\psi(t)=\frac{t^{2}}{1+t},\;t\in R_{+}.$ Then $\psi$ is monotonically increasing and continuous and $\lim\limits_{t\to\infty}\psi(t)=\infty$ with $\psi(0)=0$ so that $\psi\in\Psi$. We observe that $T$ is a $(f,g)-$weakly contractive mapping on $X$.\\
Thus $T,f$ and $g$ satisfy all the conditions of Theorem 2.1 and $T,f$ and $g$ have a unique common fixed point $\frac{2}{3}$.
\end{example}

\begin{remark}
If we replace condition (1.8.1) of Theorem 2.1 by
\begin{equation}
d(Tx,Ty)\leq\max\{d(fx,gy)-\psi(d(fx,gy)),d(gx,fy)-\psi(d(gx,fy))\},
\end{equation}
then $T,f$ and $g$ may not have common fixed points in $X$.
\end{remark}

\begin{example}
Let $X=(\frac{1}{4},1]$ with usual metric. Define $T,f$ and $g$ by
\[
T(x)=
\begin{cases}
\frac{1}{2}, & \text{if } \frac{1}{4}<x\le \frac{2}{3},\\
1-\frac{1}{2}x, & \text{if } \frac{2}{3}<x\le 1,
\end{cases}
\qquad
f(x)=
\begin{cases}
1, & \text{if } \frac{1}{3}<x\le \frac{2}{3},\\
\frac{4}{3}-x, & \text{if } \frac{2}{3}<x\le 1,
\end{cases}
\quad \text{and} \quad
g(x)=
\begin{cases}
\frac{1}{3}, & \text{if } \frac{1}{3}<x<\frac{2}{3},\\
\frac{4}{3}-x, & \text{if } \frac{2}{3}\le x\le 1.
\end{cases}
\]
Then $\overline{T(X)}\subseteq f(X),\;\overline{T(X)}\subseteq g(X)$ and $T,f$ and $g$ have no common fixed point in $X$.\\
Define $\psi:R_{+}\to R_{+}$ by $\psi(t)=\frac{t}{2},\;t\in R_{+}$.\\
\textbf{Case(i):} $x,y\in(\frac{1}{4},\frac{2}{3})$.

$d(Tx,Ty)=0$, so that $T,f$ and $g$ satisfy (2.4.1).\\
\textbf{Case(ii):} $x,y\in[\frac{2}{3},1]$.

$d(Tx,Ty)=\frac{1}{2}|y-x|$.

$d(fx,gy)-\psi(d(fx,gy))=\frac{1}{2}|y-x|$.

$d(gx,fy)-\psi(d(gx,fy))=\frac{1}{2}|y-x|$.

Thus $T,f$ and $g$ satisfy (2.4.1).\\
\textbf{Case(iii):} $x\in(\frac{1}{4},\frac{2}{3}),\;y\in(\frac{2}{3},1]$.

Now,

$d(Tx,Ty)=\frac{1}{2}(1-y)$.

Consider

$d(fx,gy)-\psi(d(fx,gy))=\frac{1}{2}(y-\frac{1}{3})$ and

$d(gx,fy)-\psi(d(gx,fy))=\frac{1}{2}(1-y)$.

We observe that $\frac{1}{4}(1-y)\leq\frac{1}{2}(y-\frac{1}{3})$, so that $T,f$ and $g$ satisfy (2.4.1).\\
\textbf{Case(iv):} $x\in(\frac{1^{2}}{4},\frac{2}{3}),\;y=\frac{2}{3}$.

$d(Tx,Ty)=0$.

Hence (2.4.1) holds.\\
\textbf{Case(v):} $x\in(\frac{2}{3},1],\;y=\frac{2}{3}$.

We now consider $d(Tx,Ty)=\frac{1}{2}(1-x)$.

$d(fx,gy)-\psi(d(fx,gy))=\frac{1}{2}(x-\frac{2}{3})$.

$d(gx,fy)-\psi(d(gx,fy))=\frac{1}{2}(x-\frac{1}{3})$.\\
Since\\
$\frac{1}{2}(1-x)\leq\max\{\frac{1}{2}(x-\frac{1}{3}),\frac{1}{2}(x-\frac{2}{3})\}=\frac{1}{2}(x-\frac{1}{3})$, for $x\in(\frac{2}{3},1]$, clearly (2.4.1) holds in this case.

Therefore the inequality (2.4.1) holds along with the other hypotheses of Theorem 2.1, and $T,f$ and $g$ have no common fixed points.\\
Here we observe that $T$ is not $(f,g)-$weakly contractive map, for take $x=\frac{3}{4},\;y=\frac{2}{3}$.\\
Then
\begin{align*}
d(Tx,Ty) &= \frac{1}{8} \not\leq \min\left\{d(fx,gy)-\psi(d(fx,gy))=\frac{1}{24},\right. \\
&\left.\quad d(gx,fy)-\psi(d(gx,fy))=\frac{5}{24}\right\}=\frac{1}{24},
\end{align*}
so that the inequality (1.8.1) does not hold.
\end{example}

This example shows that the importance of ``minimum'' in the inequality (1.8.1) of the condition `$T$ is $(f,g)-$weakly contractive mapping on $X$'.

\begin{theorem}
Let $(X,d)$ be a normed space and let $T$ be a $(f,g)-$ weakly contractive mapping. If the pairs $(T,f)$ and $(T,g)$ are weakly compatible and

$\overline{T(X)}\subseteq f(X)$, $\overline{T(X)}\subseteq g(X)$ and $f(X)$ and $g(X)$ are convex $\overline{T(X)}$ is a complete subspace of $X$, then for $x_{0}\in X$, the modified Mann iterative scheme defined by (1.13.1) with $\sum\alpha_{n}=\infty$, converges to a common fixed point of $T,f$ and $g$.
\end{theorem}

\begin{proof}
By Theorem 2.1, we obtain a common fixed point of $T,f$ and $g$, say $z$.\\
Consider the following two cases\\
\item[Case (i):] $n$ is even, we write $n=2m,\ m\in Z_{+}$.\\
Now,
\begin{align*}
||y_{n}-z|| &= ||y_{2m}-z|| \\
&= ||(1-\alpha_{2m})f(x_{2m})+\alpha_{2m}T(x_{2m})-z|| \\
&= ||(1-\alpha_{2m})(f(x_{2m})-z)+\alpha_{2m}(T(x_{2m})-z)|| \\
&\leq (1-\alpha_{2m})||f(x_{2m})-z||+\alpha_{2m}||T(x_{2m})-Tz|| \\
&\leq (1-\alpha_{2m})||f(x_{2m})-z|| \\
&\quad +\alpha_{2m}(\min\{||f(x_{2m})-gz||-\psi(||f(x_{2m})-gz||), \\
&\quad ||g(x_{2m})-fz||-\psi(||g(x_{2m})-fz||)\}) \\
&\leq (1-\alpha_{2m})||f(x_{2m})-z||+\alpha_{2m}(||f(x_{2m})-gz|| \\
&\quad -\psi(||f(x_{2m})-gz||)) \\
&= ||f(x_{2m})-z||-\alpha_{2m}\psi(||f(x_{2m})-z||).
\end{align*}
Therefore
\begin{equation}
||y_{n}-z||\leq||y_{n-1}-z||-\alpha_{n-1}\psi(||y_{n-1}-z||).
\end{equation}
Thus
\begin{equation}
||y_{n}-z||\leq||y_{n-1}-z||.
\end{equation}
\item[Case (ii):] $n$ is odd, we write $n=2m+1,\ m\in Z_{+}$.\\
Now,
\begin{align}
||y_{n}-z|| &= ||y_{2m+1}-z|| \nonumber\\
&= ||(1-\alpha_{2m+1})g(x_{2m+1})+\alpha_{2m+1}T(x_{2m+1})-z|| \nonumber \\
&= ||(1-\alpha_{2m+1})(g(x_{2m+1})-z)+\alpha_{2m+1}(T(x_{2m+1})-z)|| \nonumber \\
&\leq (1-\alpha_{2m+1})||g(x_{2m+1})-z||+\alpha_{2m+1}||T(x_{2m+1})-Tz|| \nonumber \\
&\leq (1-\alpha_{2m+1})||g(x_{2m+1})-z|| \nonumber \\
&\quad +\alpha_{2m+1}(\min\{||f(x_{2m+1})-gz||-\psi(||f(x_{2m+1})-gz||), \nonumber \\
&\quad ||g(x_{2m+1})-fz||-\psi(||g(x_{2m+1})-fz||)\}) \nonumber \\
&\leq (1-\alpha_{2m+1})||g(x_{2m+1})-z||+\alpha_{2m+1}(||g(x_{2m+1})-fz|| \nonumber \\
&\quad -\psi(||g(x_{2m+1})-fz||)) \nonumber \\
&= ||g(x_{2m+1})-z||-\alpha_{2m+1}\psi(||g(x_{2m+1})-z||). \nonumber
\end{align}
Therefore
\begin{equation}
||y_{n}-z||\leq||y_{n-1}-z||-\alpha_{n-1}\psi(||y_{n-1}-z||).
\end{equation}
Thus
\begin{equation}
||y_{n}-z||\leq||y_{n-1}-z||.
\end{equation}
We have, from (2.6.2) and (2.6.4)
\[
||y_{n}-z||\leq||y_{n-1}-z||,\ n=1,2,3,\ldots.
\]
Therefore $\{||y_{n}-z||\}$ is a decreasing sequence of nonnegative reals and hence convergent and let the converging point be $r\geq 0$.\\
If $r>0$ then $r=\inf\limits_{n}||y_{n}-z||$.\\
Thus it follows that
\[
r\leq||y_{n}-z||\ \ \text{for all } n.
\]
From (2.6.1) and (2.6.3), we get
\begin{equation}
||y_{n}-z||\leq||y_{n-1}-z||-\alpha_{n-1}\psi(||y_{n-1}-z||),\ n\geq 0.
\end{equation}
Now,
$\psi(r)\leq\psi(||y_{n}-z||)$, which implies that
\begin{equation}
\alpha_{n}\psi(r)\leq\alpha_{n}\psi(||y_{n}-z||).
\end{equation}
Summing both sides of (2.6.6) from $n=0$ to $m$, we get
\begin{align*}
\sum_{n=0}^{m}\alpha_{n}\psi(r) &\leq \sum_{n=0}^{m}\alpha_{n}\psi(||y_{n}-z||) \\
&\leq \sum_{n=0}^{m}(||y_{n}-z||-||y_{n+1}-z||) \\
&= ||y_{0}-z||-||y_{m+1}-z|| \\
&< ||y_{0}-z||,
\end{align*}
a contradiction.

Therefore $r=0$.

Hence $y_{n}\to z$ as $n\to\infty$.

\end{proof}

As the proof of the following theorem is similar to that of Theorem 2.6, we state the theorem without proof.

\begin{theorem}
Let $T$ be a $(f,g)-$ weakly contractive mapping on a normed space $X$. Assume that $\overline{T(X)}\subseteq f(X)$ and $\overline{T(X)}\subseteq g(X)$. If $f(X)$ and $g(X)$ are convex and $\overline{T(X)}$ is a complete subspace of $X$, and the pairs $(T,f)$ and $(T,g)$ are weakly compatible, then for $x_{0}\in X$, the modified Ishikawa iterative scheme defined by (1.14.1) with $\sum\alpha_{n}\beta_{n}=\infty$, converges to a common fixed point of $T,f$ and $g$.
\end{theorem}

\section{Invariant Approximation Results}

Meinardus \cite{Meinardus1963} introduced the notion of invariant approximation. Brosowski \cite{Brosowski1969} initiated the study of invariant approximation using fixed point theory. In 2006, Beg and Abbas \cite{Beg2006} proved the existence of common fixed points for two selfmaps in the set of best approximations.

\begin{definition}
Let $M$ be a subset of a metric space $X$. The set $P_{M}(u)=\{x\in M:d(x,u)=\text{dist}(u,M)\}$ is called the set of best approximations to $u$ in $X$ out of $M$, where $\text{dist}(u,M)=\inf\{d(y,u):y\in M\}$.
\end{definition}

The following two lemmas are useful in our subsequent discussion.

\begin{lemma}
Let $(X,d)$ be a metric space and let $x_{0}\in X-M$. If $M$ is a compact subset of $X$ then $P_{M}(x_{0})\neq\emptyset$.
\end{lemma}

\begin{lemma}
The set of best approximations to $x_{0}$ in $X$ out of $M$ is closed, i.e., $P_{M}(x_{0})$ is closed, where $M$ is a subset of a metric space $X$.
\end{lemma}

\begin{theorem}
Let $X$ be a metric space. Suppose that $T,f$ and $g$ are selfmaps of $X$, $T$ is continuous, $T$ is $(f,g)$-weakly contractive mapping and the pairs $(T,f)$ and $(T,g)$ are weakly compatible. Assume that $F(T)\cap F(f)\cap F(g)\neq\emptyset.$ Further, if either (i) $g$ and $T$ leaves $f-$invariant compact subset $M$ of $X$ as invariant with $f(P_{M}(x_{0}))\subseteq P_{M}(x_{0})$ (or) (ii) $f$ and $T$ leaves $g-$invariant compact subset $M$ of $X$ as invariant with $g(P_{M}(x_{0}))\subseteq P_{M}(x_{0})$, then $P_{M}(x_{0})\cap F(T)\cap F(f)\cap F(g)\neq\emptyset.$
\end{theorem}

\begin{proof}
Since $M$ is compact and by Lemma 3.2, we have $P_{M}(x_{0})\neq\emptyset$.

We now show that $T(P_{M}(x_{0}))\subseteq f(P_{M}(x_{0}))$ and $T(P_{M}(x_{0}))\subseteq g(P_{M}(x_{0}))$. Otherwise, there exist $a,b\in P_{M}(x_{0})$ such that $Ta\notin f(P_{M}(x_{0}))$ and $Tb\notin g(P_{M}(x_{0}))$.\\
Observe that $x_{0}\neq fa$ and $x_{0}\neq gb$.\\
For, if $x_{0}=fa$ then $fa\in F(T)\cap F(f)\cap F(g)$.\\
Since $a\in P_{M}(x_{0})$ we have $f(a)\in f(P_{M}(x_{0}))$, i.e., $x_{0}\in f(P_{M}(x_{0}))$.\\
Hence
\begin{align*}
0 &= d(x_{0},fa) = d(x_{0},M) \leq d(x_{0},Ta) = d(Tx_{0},Ta) \\
&\leq \min\{d(fx_{0},ga)-\psi(d(fx_{0},ga)),d(gx_{0},fa)-\psi(d(gx_{0},fa))\} \\
&\leq d(x_{0},fa)-\psi(d(x_{0},fa)) = 0.
\end{align*}
Thus $d(x_{0},Ta)=0$ implies $Ta=x_{0}=f(a)\in f(P_{M}(x_{0}))$, a contradiction to our assumption.\\
Hence $x_{0}\neq fa$.\\
Now,
\begin{align*}
d(x_{0},fa) &= d(x_{0},M) \leq d(x_{0},Ta) = d(Tx_{0},Ta) \\
&\leq \min\{d(fx_{0},ga)-\psi(d(fx_{0},ga)),d(gx_{0},fa)-\psi(d(gx_{0},fa))\} \\
&\leq d(x_{0},fa)-\psi(d(x_{0},fa)) < d(x_{0},fa),
\end{align*}
a contradiction. Hence $T(P_{M}(x_{0}))\subseteq f(P_{M}(x_{0}))$.\\
Similarly we can prove that $x_{0}\neq g(b)$ and $T(P_{M}(x_{0}))\subseteq g(P_{M}(x_{0}))$.\\
Since $P_{M}(x_{0})$ is closed subset of a compact set $M$, we have $T(P_{M}(x_{0}))$ is compact and hence $T(P_{M}(x_{0}))$ is closed and complete.
So by Theorem 2.1, $T$, $f$ and $g$ have a unique common fixed point in $P_{M}(x_{0})$.

Hence \(P_{M}(x_{o})\cap F(T)\cap F(f)\cap F(g)\neq\emptyset\).

\end{proof}

\begin{theorem}[3.5]
Let \(X\) be a metric space. Suppose that \(T, f\) and \(g\) are selfmaps of \(X\), \(T\) is continuous, \(T\) is \((f,g)\)-weakly contractive mapping and the pairs \((T,f)\) and \((T,g)\) are weakly compatible. Assume that either

(i) \(g\) and \(T\) leave \(f\)-invariant compact subset \(M\) of \(X\) as invariant with \(f(P_{M}(u))\subseteq P_{M}(u)\) for any \(u\in X\)

(or)

(ii) \(f\) and \(T\) leave \(g\)-invariant compact subset \(M\) of \(X\) as invariant with \(g(P_{M}(u))\subseteq P_{M}(u)\) for any \(u\in X\). \\
Further, if for each \(a\in P_{M}(u), d(x,Ta) < d(x,fa)\) and \(d(x,Ta) < d(x,ga)\) for each \(x\in X\), then \(T,f\) and \(g\) have a unique common fixed point in \(P_{M}(u)\) which is a best approximation of \(u\) in \(M\).
\end{theorem}

\begin{proof}
For each \(u\in X, P_{M}(u)\neq\emptyset\) and compact. If \(u\in M\) we have \(P_{M}(u)=\{u\}\). Then by hypotheses \(f(u)=u\) and by (i) \(g(u)=u, T(u)=u\) so that the theorem follows.\\
Thus, without loss of generality, we assume that \(u\in X-M\).

Since \(P_{M}(u)\neq\emptyset\) and compact, by (i) we have \(g(P_{M}(u))\subseteq P_{M}(u)\) and \(T(P_{M}(u))\subseteq P_{M}(u)\). Thus \(T,f\) and \(g\) are selfmaps of \(P_{M}(u)\).

We now show that \(T(P_{M}(u)) = T(P_{M}(u)) \subseteq f(P_{M}(u))\) and \(\overline{T(P_{M}(u))} = T(P_{M}(u)) \subseteq g(P_{M}(u))\).\\
Otherwise, there exist \(a,b\in P_{M}(u)\) such that \(Ta\notin f(P_{M}(u))\) and \(Tb\notin g(P_{M}(u))\).\\
Now,
\[
d(u,fa) = d(u,M) \leq d(u,Ta) < d(u,fa) = d(u,M),
\]
a contradiction. Also
\[
d(u,gb) = d(u,M) \leq d(u,Tb) < d(u,gb) = d(u,M),
\]
a contradiction.
Hence \(T(P_{M}(u))\subseteq f(P_{M}(u))\) and \(T(P_{M}(u))\subseteq g(P_{M}(u))\).

Thus \(T, f\) and \(g\) satisfy all the hypotheses of Theorem 2.1 on \(P_{M}(u)\), so that by Theorem 2.1, \(T,f\) and \(g\) have a unique common fixed point in \(P_{M}(u)\).

Hence the theorem follows.
\end{proof}

\section{ Conclusion}

In this work, we introduced and investigated a new class of mappings, termed \((f,g)\)-weakly contractive mappings, which extend the notions of weakly contractive and weakly compatible selfmaps to a triad of mappings \(T\), \(f\), and \(g\) on a metric space. Under appropriate assumptions on convexity, completeness, and weak compatibility, we established a common fixed point theorem guaranteeing the existence and uniqueness of a common fixed point of \(T\), \(f\), and \(g\).

Furthermore, we demonstrated that the modified Mann and modified Ishikawa iterative schemes associated with \(T\), \(f\), and \(g\) converge to this unique common fixed point in normed spaces. In addition, invariant approximation results were obtained, showing that the common fixed point can also be realized as a best approximation from an invariant subset.

The results presented here generalize and unify several classical and recent fixed point theorems, including those of Beg and Abbas~\cite{Beg2006} and Azam and Shakeel~\cite{Azam2008}, and provide a broader framework for studying weakly contractive-type mappings. Future research directions include extending these results to more generalized settings such as \(b\)-metric, partial metric, or modular function spaces, and exploring their applications to nonlinear integral and differential equations.

\end{document}